\documentclass[12pt,twoside,reqno]{amsart}
\usepackage[colorlinks=true,citecolor=blue]{hyperref}
\usepackage{mathptmx, amsmath, amssymb, amsfonts, amsthm, mathptmx, enumerate, color}
\usepackage{graphicx}
\usepackage{epstopdf}

\newtheorem{theorem}{Theorem}[section]
\newtheorem{corollary}{Corollary}[section]
\newtheorem{lemma}{Lemma}[section]
\newtheorem{proposition}{Proposition}[section]
\theoremstyle{definition}
\newtheorem{definition}{Definition}[section]

\newtheorem{remark}{Remark}[section]

\numberwithin{equation}{section}
\begin{document}
\setcounter{page}{1}

\vspace*{1.0cm}
\title[Directional Differentiability of Metric Projection ]
{Directional Differentiability of the Metric Projection Operator in Hilbert Spaces and Hilbertian Bochner Spaces}
\author[J. Li, L. Cheng, L. Liu, L. Xie]{Jinlu Li$^{1}$, Li Cheng$^{2}$, Lishan Liu$^{3}$, Linsen Xie$^{2,*}$}
\maketitle
\vspace*{-0.6cm}

\begin{center}
{\footnotesize {\it

$^1$Department of Mathematics, Shawnee State University, Portsmouth, Ohio 45662 USA\\
$^2$Department of Mathematics, Lishui University, Lishui 323000, Zhejiang, People's Republic of China\\
$^3$School of Mathematical Sciences, Qufu Normal University, Qufu 273165, Shandong,\\ People's Republic of China

}}\end{center}

\vskip 4mm {\small \noindent {\bf Abstract.}
Let $H$ be a real Hilbert space and $C$ a nonempty closed and convex subset of $H$. Let $P_C: H\rightarrow C$ denote the (standard) metric projection operator. In this paper, we study the G$\hat{a}$teaux directional differentiability of $P_C$ and investigate some of its properties. The G$\hat{a}$teaux directionally derivatives of $P_C$ are precisely given for the following cases of the considered subset $C$: 1. closed and convex subsets; 2. closed balls; 3. closed and convex cones (including proper closed subspaces). For special Hilbert spaces, we consider directional differentiability of $P_C$ for some Hilbert spaces with orthonormal bases and the real Hilbert space $L^2([-\pi,\pi])$ with the trigonometric orthonormal basis. \\

 }

\renewcommand{\thefootnote}{}
\footnotetext{ $^*$Corresponding author.
\par
E-mail addresses:jli@shawnee.edu(J.Li), chenglilily@126.com(L.Cheng),
mathlls@163.com(L.Liu),\\
 linsenxie@163.com(L.Xie).
\par
Received August 10, 2023; Accepted September 12, 2023. }

\section{Introduction and preliminaries}
\noindent{{{1.1. \bf The convexity and smoothness of Hilbert spaces.}}} Let $H$ be a real Hilbert space with norm $\|\cdot\|$ and inner product $\langle\cdot,\cdot\rangle$. We list some notations below used in this paper. For any $u,v \in H$ with $u \neq v$, we write

(a) $\overline{v,u}=\{tv+(1-t)u: 0\leq t\leq1\}$, a closed line segment in $H$ with ending points at $u$ and  $v;$

(b) $\overrightarrow{v,u} = \{v+t(u-v): 0\leq t<\infty\},$ a closed ray in $H$ with ending point at $v$ and with direction $u-v;$

(c)  $ \overleftrightarrow{v, u} = \{v+t(u-v): -\infty < t<\infty\},$ a line in $H$ passing through points $v$ and $u$;

(d) $B = \{x\in H:\ \|x\|\leq1\},$ the closed unit ball in $H$;

(e) $S = \{x\in H:\ \|x\|=1\},$ the (closed) unit sphere in $H. $

Hilbert space $H$ is a special case of uniformly convex and uniformly smooth Banach space. Let $\delta$ denote the modules of convexity of $H$, which is the function $\delta: [0, 2] \rightarrow [0, 1]$ defined by
$$\delta(\epsilon) =\inf\left\{1-\|\frac{x+y}{2}\|:x,y\in S,\|x-y\|\geq\epsilon\right\},\ \text{for any} \ \epsilon \in [0, 2].$$
The modules of convexity $\delta$ of a Hilbert space $H$ has the following analytic representation (see Lemma 4.2 in [3] by Baumeister).
\begin{equation}\label{1.1} \delta(\epsilon) = 1- \sqrt{1-\frac{1}{4}\epsilon^2},\ \text{for any} \ \epsilon \in [0, 2]. \end{equation}
Let $\rho$ denote the modules of smoothness of the Hilbert space $H$, defined by
$$\rho(t) = \sup \left\{\frac{\|x+y\|+\|x-y\|}{2}-1: x\in S,\ y\in H, \|y\|=t\right\},\ \text{for} \ t>0.$$
The modules of smoothness $\rho$ of Hilbert space $H$ has the following analytic representation (see Corollary 4.17 in [3] by Baumeister):
\begin{equation}\label{1.2}\rho(t) = \sqrt{1+t^2} -1,\ \text{for} \ t > 0.\end{equation}
The norm $\|\cdot\|$ of $H$ is G$\hat{a}$teaux directionally differentiable. More precisely speaking, that is
\begin{equation}\label{1.3}\lim_{t\downarrow0}\frac{\|x+tv\| -\| x\|}{t } = \langle x,v\rangle,\ \text{for any} \ (x,v)\in S\times S.\end{equation}

\noindent{{{1.2. \bf The metric projection onto closed and convex subsets in Hilbert spaces.}}}
Let $C$ be a nonempty closed and convex subset of $H$ and let $P_C: X\rightarrow C$ denote the (standard) metric projection operator. That is, for any $x\in H$, $P_Cx\in C$ such that
\begin{equation}\label{1.4}  \|x-P_Cx\|\leq\| x-z\|,\ \text{for all}\ z\in C.\end{equation}
$P_Cx$ is called the metric projection of point $x$ onto $C$. $P_Cx$ is considered as the best approximation of $x$ by elements of $C$, which is the nearest point from $x$ to $C$. The distance function on $H$ with respect to $C$ is defined by
$$d(x, C) =\| x-P_Cx\|, \ \text{for any}\ x\in H.$$
In uniformly convex and uniformly smooth Banach spaces, the metric projection $P_C$ has many useful and well-known properties, which are particular hold for Hilbert spaces. We list some of them below as a proposition without proof. See [1, 11] for more details and proofs.

\begin{proposition}
	Let $C$ be a nonempty closed and convex subset of a Hilbert space $H$ and let $P_C: X\rightarrow C$ be the metric projection. Then, we have
	
	(i)\ 	Every nonempty closed and convex subset $C$ in $H$ is a Chebyshev set in $H$. That is, for  any $x\in H$, there is a unique point $P_Cx\in C$ satisfying \eqref{1.4};
	
	(ii)\ The basic variational principle: for any $x\in H$ and $u\in C$,
	\begin{equation}\label{1.5} u = P_Cx \Leftrightarrow \langle x-u, u-z\rangle\geq 0,\ \text{ for all}\ z\in C;       \end{equation}
	
	(iii)\ For any $x\in H$ and $u\in C$,
	\begin{equation}\label{1.6}    u = P_Cx  \Leftrightarrow \langle x-u, u-z\rangle \geq\|x-u\|^2, \ \text{ for all}\ z\in C;       \end{equation}
	
	(iv)\	$P_C $ is (strongly) monotone, that is
	\begin{equation}\label{1.7} \langle P_C x-P_C y, x-y \rangle \geq\|P_Cx-P_Cy\|^2, \ \text{ for all}\ x, y\in H;        \end{equation}
	
	(v)\	Nonexpansiveness:
	\begin{equation}\label{1.8} \|P_Cx-P_Cy\| \leq\|x-y\|,\ \text{ for any} \ x, y\in H;                               \end{equation}
	
	More precisely speaking (see Problem 7 on page 114 in [3]), for any $x, y \in H$,
	$$\| P_Cx-P_Cy \|<\| x-y \| \ \text{ or}  \  P_Cx-P_Cy = x-y;$$
	
	(vi)\	The distance function $d(\cdot, C)$ is continuous on $H$.
\end{proposition}

\noindent{{{1.3. \bf The G$\hat{\mathbf{a}}$teaux directional differentiability of the metric projection operator in Hilbert spaces.}}}
In Hilbert spaces, the G$\hat{a}$teaux directional differentiability of the metric projection operator has been studied by many authors (see [7-9, 23]. It has been applied to programming problems, optimal control problems, and so forth (see [14-15, 19]). We recall the definition, which is adopted by most of authors.

Let $C$ be a nonempty closed and convex subset of a Hilbert space $H$. For $x\in H$ and $v\in H$ with $v \neq \theta$, if the following limit exists, which is a point in $H$,
\begin{equation}\label{1.9}     \lim_{t\downarrow0}{\frac{P_C(x+tv)-P_C(x)}{t}}\end{equation}
then, $P_C$ is said to be (G$\hat{a}$teaux) directionally differentiable at point $x$ along direction $v$.

The concept of the G$\hat{a}$teaux directional differentiability of the metric projection operator has been extended from Hilbert spaces to Banach spaces and more general normed spaces (see [2, 4, 6, 14-20, 22]). In the case of Hilbert spaces, the basic variational principle and the nonexpansive property of $P_C$ play very important roles in the study of the G$\hat{a}$teaux directional differentiability of the metric projection operator. Since in the case of Banach spaces, the metric projection does not enjoy such basic variational principle (part (ii) in Proposition 1.1) and the nonexpansive property (part (iii) in Proposition 1.1), it will be more difficult to study the G$\hat{a}$teaux directional differentiability of the metric projection operator in Banach spaces.

In [12], the G$\hat{a}$teaux directional differentiability of the metric projection operator was studied in uniformly convex and uniformly smooth Banach spaces. The analytic solutions of the G$\hat{a}$teaux directionally derivatives of the metric projection operator are proved for some special closed convex subsets, such as closed balls, cones and subspaces. In this paper, we will investigate similar analytic solutions of the G$\hat{a}$teaux directionally derivatives of the metric projection operator in Hilbert spaces. Since Hilbert spaces are special cases of uniformly convex and uniformly smooth Banach spaces, the G$\hat{a}$teaux directionally derivatives of the metric projection operator are precisely presented.

\section{Properties of inverse images of the metric projection operator in Hilbert spaces}
Throughout this section, let $H$ be a Hilbert space and $C$ a nonempty closed and convex subset of $H$. For any $y\in C$, the inverse image of $y$ by the metric projection operator $P_C$ in $H$ is defined by
$$P_C^{-1}(y)=\{x\in H: P_C(x)=y\}.$$
The inverse mapping $P_C^{-1}$ has many useful properties and it plays a crucial role in the study of the G$\hat{a}$teaux directional differentiability of the metric projection operator in Hilbert spaces. In [11], many properties of the inverse images of the metric projection $P_C$ in uniformly convex and uniformly smooth Banach spaces have been proved, which were used in the studying of the G$\hat{a}$teaux directional differentiability of $P_C$ in uniformly convex and uniformly smooth Banach spaces in [12]. For Hilbert spaces, as special cases of uniformly convex and uniformly smooth Banach spaces, $P_C^{-1}$ has some special properties.

\begin{lemma}
	Let $C$ be a nonempty closed and convex subset of $H$. Then,
	
	(i) \ $H =\underset{y\in C}\bigcup{P_C^{-1}(y)},$
	
	(ii)\ For any $y\in C$, $P_C^{-1}(y)$ is either a singleton $\{y\}$, or a closed and convex cone in $H$ with  vertex at $y$.
\end{lemma}

\begin{proof}  By part (i) in Proposition 1.1, part (i) of this lemma is proved. For any $y\in C$, if $P_C^{-1}(y)$ is either a singleton $\{y\}$, then, by Theorem 3.1 in [11], $P_C^{-1}(y)$ is a closed cone with vertex at $y$ in $H$. We only need to prove that $P_C^{-1}(y)$ is convex. To this end, for any $x_1, x_2\in P_C^{-1}(y)$, for any $\alpha\in[0,1]$ and for any $z\in C$, we have
$$ \langle(\alpha x_1 + (1-\alpha)x_2)-y,y-z\rangle = \alpha\langle x_1-y,y-z\rangle+(1-\alpha)\langle x_2-y,y-z\rangle.$$
Then, the convexity of $P_C^{-1}(y)$ follows from the basic variational principle of $P_C$ in Hilbert space $H$ immediately.

Now we prove the closeness of $P_C^{-1}(y)$. Suppose $\{x_n\}\subseteq P_C^{-1}(y)$ and $x_n\rightarrow x,$ as  $n\rightarrow\infty.$ By the basic variational principle \eqref{1.5} in Proposition 1.1, for every $n$ and for every $z\in C$, we have
$$ \langle x_n-y, y-z\rangle \geq 0.$$
Since $x_n\rightarrow x,$ as $n\rightarrow\infty,$ this implies
$$ \langle x-y, y-z\rangle\geq0, \ \text{for all} \ z\in C.$$
By the basic variational principle \eqref{1.5} in Proposition 1.1 again, this implies $x\in P_C^{-1}(y)$.\end{proof}

However, in Theorem 3.1 in [11], it was proved that, in uniformly convex and uniformly smooth Banach spaces, $P_C^{-1}(y)$ is not convex, in general. Based on the two cases in part (ii) of Lemma 2.1, we classify the points in $C$ to two groups, which is defined in [12] for uniformly convex and uniformly smooth Banach spaces.

\begin{definition}
	 Let $C$ be a nonempty closed and convex subset in $H$. Let $y \in C$.
	
	(i)\ \ If $P_C^{-1}(y) ={y}$, then $y$ is called an internal point of $C$;
	
	(ii)\ If $P_C^{-1}(y) \supsetneq {y}$, then $y$ is called a cuticle point of $C$.
\end{definition}

The collection of all internal points of $C$ is denoted by $\mathbb{I}(C)$ and the collection of all cuticle points of $C$ is denoted by $\mathbb{C}(C)$. Then $\{\mathbb{I}(C), \mathbb{C}(C)\}$ forms a partition of $C$. More precisely, we have
$$\mathbb{I}(C)\cap\mathbb{C}(C)=\emptyset \ \text{and} \   \mathbb{I}(C)\cup\mathbb{C}(C) = C.$$

By the definitions of $\mathbb{I}(C)$ and $\mathbb{C}(C)$, they have the following characteristics.

(a)	$y\in\mathbb{I}(C)$ if and only if, for $u \in H$,
$$y=P_C(u+y)   \Rightarrow   u = \Theta;$$

(b)	$y\in\mathbb{C}(C)$ if and only if, there is $u\in H$ with $u \neq \Theta$ such that
$$y=P_C(u+y).$$

For an arbitrary fixed $r>0$, we respectively write the closed, open balls and sphere in $H$ with radius $r$ at center $c\in H$ by $B(c, r)$, $B^o(c, r)$ and $S(c, r)$. Then $B(c, r)$ is a closed and convex subset of $H$. In particular, as defined in section 1, when $c = \theta$ and $r = 1$, the unit closed ball and unit sphere are written by $B(\theta,1) = B$ and $S(\theta,1) = S.$ Since every Hilbert space is a uniformly convex and uniformly smooth Banach space, then, by the results in section 3 in [12], we respectively have

\begin{lemma}
(A new version of Lemma 3.5 in [12]).  Let $c\in H$ and $r>0.$ Then we have
\begin{equation}\label{2.1}
	P_{B(c,r)}(x) = c+\frac{r}{\|x-c\|}(x-c),  \ \text{for any} \ x\in H\backslash B(c, r).
\end{equation}
In particular, we have
$$P_B(x) = \frac{x}{\|x\|},  \ \text{for any} \ x\in H \ \text{with} \ \|x\|>1.$$
\end{lemma}

\begin{proposition}
	 (A new version of Proposition 3.6 in [12]). Let $c\in H$ and $r>0$, we have
	
	 (i) \ $\mathbb{I}(B(c, r)) = B^o(c, r);$
	
	 (ii)\ $\mathbb{C}(B(c, r)) = S(c, r);$
	
	 (iii)\ For any $y\in S(c, r)$,
	 \begin{equation}\label{2.2} P_{B(c, r)}^{-1}(y) = \{y+t(y-c): 0 \leq t<\infty\}. \end{equation}
\end{proposition}

\begin{proposition}
 (A new version of Proposition 3.7 in [12]). Let $D$ be a proper closed subspace of a Hilbert space $H$. Then, we have

 (i)\ \	$\mathbb{I}(D) = \emptyset;$

 (ii)\	$\mathbb{C}(D) = D.$
\end{proposition}

\begin{definition}
	Let $D$ be a proper closed subspace of a Hilbert space $H$. We denote
	$$D^\bot = \{x\in H:\langle x,z\rangle = 0,\ \text{for all} \ z\in D\}.$$
	Then, $D^\bot$ is a closed and convex cone in $H$ with vertex at $\theta$. $D^\bot$ is called the orthogonal cone of $D$ in this Hilbert space $H$.
\end{definition}

However, notice that, in uniformly convex and uniformly smooth Banach spaces, by Corollary 2.8 in [11], $D^\bot$ is a closed cone, which is non-convex, in general.

\begin{proposition}
(A new version of Proposition 3.9 in [12]). Let $D$ be a proper closed subspace of a Hilbert space $H$. Then, we have

(i)\ $P_D^{-1}(\theta) = D^\bot;$

(ii)\ $y \in D  \Rightarrow P_D^{-1}(y) = y + P_D^{-1}(\theta) = y + D^\bot.$
\end{proposition}

\begin{definition}
Let $K$ be a closed and convex cone in a Hilbert space $H$ with vertex at $v \in X$. We denote
$$K^\wedge=P_K^{-1}(v)=\{x\in X: \langle x-v,v-z\rangle\geq 0, \ \text{for all} \ z\in K\}.$$
Then by part (ii) in Lemma 2.1, $K^\wedge=P_K^{-1}(v)$ is a closed and convex cone in $X$ with vertex at $v$. $K^\wedge$ is called the dual cone of $K$ in $H$.
\end{definition}

However, notice that, in uniformly convex and uniformly smooth Banach spaces, by Theorem 3.1 in [11], $P_K^{-1}(v)$ is a closed cone in $X$ with vertex at $v$ and it is non-convex, in general.

\begin{lemma}
	 (A new version of Lemma 3.11 in [12]).  Let $K$ be a closed and convex cone in $H$ with vertex at $v\in H$. Then, for every $y\in K$ with $y\neq v$ and for any $u=v+t(y-v)$ with $t>0$, we have
	 $$P_K^{-1}(u) = u-y+P_K^{-1}(y).$$
\end{lemma}

\begin{corollary}
	(A new version of Corollary 3.12 in [12]).  Let $K$ be a closed and convex cone in $H$ with vertex at $\theta$. Then, for every $y\in K$ with $y\neq \theta$, we have
	$$y+P_K^{-1}(ty)=ty+P_K^{-1}(y),\ \text{for any} \ t>0.$$
\end{corollary}

\section{Directional differentiability of the metric projection operator onto closed and convex subsets}
In this section, let $H$ be a Hilbert space and $C$ a nonempty closed and convex subset of $H$. By part (i) in Proposition 1.1, the metric projection $P_C$ is a well-defined single-valued mapping from $H$ onto $C$. By applying this property of $P_C$, it is possible to define the G$\hat{a}$teaux directionally differentiability of $P_C$.
\begin{definition}
Let $C$ be a nonempty closed and convex subset of $H$. For $x\in H$ and $v\in H$ with $v \neq \theta$, if the following limit exists, which is a point in $H$,
\begin{equation}\label{3.1}
	\lim_{t\downarrow0}{\frac{P_C(x+tv)-P_C(x)}{t}}
\end{equation}
then, $P_C$ is said to be (G$\hat{a}$teaux) directionally differentiable at point $x$ along direction $v$, which is denoted by
\begin{equation}\label{3.2}
	P'_C(x;v) = \lim_{t\downarrow0}{\frac{P_C(x+tv)-P_C(x)}{t}}.                                              \end{equation}
$P'_C(x;v)$ is called the (G$\hat{a}$teaux) directional derivative of $P_C$ at point $x$ along direction $v$; and $v$ is called a (G$\hat{a}$teaux) differentiable direction of $P_C$ at $x$. If $P_C$ is (G$\hat{a}$teaux) directionally differentiable at point $x\in H$ along every direction $v\in H$ with $v \neq \theta$, then $P_C$ is said to be (G$\hat{a}$teaux) directionally differentiable at point $x\in H$. It is denoted by
\begin{equation}\label{3.3}
	P'_C(x)(v) = \lim_{t\downarrow0}{\frac{P_C(x+tv)-P_C(x)}{t}}, \ \text{for}\ v\in H \ \text{with} \ v \neq \theta.
\end{equation}
$P'_C(x)(v)$ is called the (G$\hat{a}$teaux) directional derivative of $P_C$ at point $x$ along direction $v$. Let $A$ be an open subset in $H$. If $P_C$ is (G$\hat{a}$teaux) directionally differentiable at every point $x\in A$, then $P_C$ is said to be (G$\hat{a}$teaux) directionally differentiable on $A \subseteq H$.
\end{definition}
Notice that the (G$\hat{a}$teaux) directional differentiability of $P_C$ at point $x$ along direction $v$ can be equivalently defined by: there is a point, denoted by $P'_C(x;v) \in H$, such that
$$P_C(x+tv) = P_C(x) + tP'_C(x;v)+o(t),\ \text{for} \ t > 0.$$

Following the results in section 4 in [12], we give the solutions of the directionally derivatives of the metric projection $P_C$ onto nonempty closed and convex subsets of $H$.
\begin{proposition}
	 Let $C$ be a nonempty closed and convex subset of $H$. Then, $P_C$ is directionally differentiable at every point $x\in H\backslash C$ along both directions $P_C(x)-x$ and $x-P_C(x)$, respectively, such that
	$$P'_C(x; x-P_C(x)) = P'_C(x; P_C(x)-x) = \theta, \ \text{for every}\ x\in H\backslash C.$$
\end{proposition}
\begin{lemma}
	Let $C$ be a nonempty closed and convex subset of $H$. Let $u, w \in C$ with $u \neq w$. For every point $x\in\overline{u,w}$, we have
	
	(a)	If $x \neq u$ and $x \neq w$, then $P_C$ is directionally differentiable at point $x$ along both directions  $w-u$ and $u-w$ such that
	$$P'_C(x; u-w) = u-w    \quad \text{and} \quad     P'_C(x; w-u) = w-u;$$
	
	(b)	$P_C$ is directionally differentiable at $u$ along both directions $w-u$ and $u-w$ such that
	$$P'_C(u; w-u) = w-u     \quad \text{and} \quad      P'_C(w; u-w) = u-w.$$
\end{lemma}
\begin{lemma}
	Let $C$ be a nonempty closed and convex subset of $H$. Then, the following statements are equivalent
	
	(i) $P_C$ is directionally differentiable on $H$ such that, for every point $x\in H$,
	$$P'_C(x)(v)=\theta, \ \text{for any} \ v\in H \ \text{with} \ v \neq \theta;$$
	
	(ii) $P_C$ is a constant operator; that is, $C$ is a singleton.
\end{lemma}
\begin{corollary}
 Let $C$ be a nonempty closed and convex subset of $H$. Then, for every point $x\in H$, there is least one differentiable direction of $P_C$ at $x$.
\end{corollary}
Next lemma proves that the directionally derivative of $P_C$ is positive homogenous.
\begin{lemma}
	Let $C$ be a nonempty closed and convex subset of $H$. For $x\in H$ and $v\in H$ with $v \neq \theta$, if $P_C$ is directionally differentiable at $x$ along direction $v$, then, for any $\lambda > 0$, $P_C$ is directionally differentiable at $x$ along direction $\lambda v$ such that
	$$P'_C(x; \lambda v) = \lambda P'_C(x;v), \ \text{for any} \ \lambda > 0.$$
\end{lemma}
\begin{proposition}
	Let $C$ be a nonempty closed and convex subset of $H$. Let $y\in C$. Suppose $(P_C^{-1}(y))^o \neq\emptyset.$ Then, $P_C$ is directionally differentiable on $(P_C^{-1}(y))^o$ such that, for any $x\in(P_C^{-1}(y))^o$, we have
	$$P'_C(x)(v) = \theta, \ \text{for every} \ v\in H \ \text{with}\  v \neq \theta.$$
\end{proposition}
\begin{proposition}
Let $C$ be a nonempty closed and convex subset of $H$. Suppose $C^o \neq\emptyset$. Then $P_C$ is directionally differentiable on $C^o$ such that, for any $x\in C^o$, we have
$$P'_C(x)(v) = v, \ \text{for every} \ v\in H \ \text{with}\ v \neq \theta.$$
\end{proposition}

\section{Directional differentiability of the metric projection operator onto closed balls}
Recall that, for $c\in H$ and $r > 0$, the closed, open balls and sphere in $H$ with center at $c$ and with radius $r$ are respectively written as $B(c, r)$, $(B(c, r))^o$ and $S(c, r)$.

For any $x\in S(c, r)$, we define two subsets $x_{(c,\ r)}^\uparrow$ and $x_{(c,\ r)}^\downarrow$ of $H\backslash\{\theta\}$ as follows: for $v\in H$ with $v \neq \theta$, we say

(a)	$v\in x_{(c,\ r)}^\uparrow  \Leftrightarrow  \ \text{there is} \ \delta > 0 \ \text{such that} \ \|(x+tv)-c\| \geq r, \ \text{for all} \ t\in (0, \delta);$

(b)	$v\in x_{(c,\ r)}^\downarrow  \Leftrightarrow  \ \text{there is} \ \delta > 0 \ \text{such that} \ \|(x+tv)-c\| < r, \ \text{for all}\ t\in(0, \delta).$

In particular, $B$ is the closed unit ball and $S$ is the closed unit sphere in $H$. For any given $x\in S$ and for $v\in H$ with $v \neq \theta$, we write

(c) $x^\uparrow = x_{(\theta, 1)}^\uparrow:  v\in x^\uparrow \Leftrightarrow \ \text{there is} \ \delta > 0 \ \text{such that} \ \|x+tv\| \geq 1, \ \text{for all} \ t\in(0, \delta);$

(d) $x^\downarrow = x_{(\theta, 1)}^\downarrow:  v\in x^\downarrow\Leftrightarrow  \ \text{there is} \ \delta > 0 \ \text{such that} \  \|x+tv\| < 1, \ \text{for all} \ t\in(0, \delta).$

The Lemma 5.1 in [12] shows that, for any given $x\in S(c, r)$, the two subsets $x_{(c,\ r)}^\uparrow$ and $x_{(c,\ r)}^\downarrow$ form a partition of $H\backslash\{\theta\}$.

\begin{lemma}
	 Let $c\in H$ and $r>0$. Then, for any $x\in S(c, r)$, we have
	
	$$x_{(c,\ r)}^\uparrow\cap x_{(c,\ r)}^\downarrow = \emptyset  \ \text{and} \   x_{(c,\ r)}^\uparrow\cup x_{(c,\ r)}^\downarrow = H\backslash\{\theta\}.$$
\end{lemma}

Next, we apply Theorem 5.2 in [12] to give the analytic representations (solutions) of the derivatives of the metric projection operator on closed balls in Hilbert spaces.

In [12], the concept of the function of smoothness of smooth Banach spaces is introduced. Let $X$ be a smooth Banach space. Define the function of smoothness of smooth of $X, \psi: S\times S
\rightarrow \mathbb{R}_+$ by
$$\psi(x,v) = \lim_{t\downarrow0}\frac{\|x+tv\|-\|x\|}{t},  \ \text{for any} \ (x,v)\in S\times S.$$
In particular, when the considered smooth Banach space is a Hilbert space, by \eqref{1.3}, we have
$$\psi(x,v) = \langle x,v\rangle,\ \text{for any} \ (x,v)\in S\times S.$$

In Theorem 5.2 in [12], $\psi$ is used to represent the solution of the directionally derivatives of the metric projection operator onto closed balls in uniformly convex and uniformly smooth Banach spaces. We recall Theorem 5.2 in [12] below.

\bigskip

\noindent{\bf Theorem 5.2 in [12].} {\it Let $C = B(c, r)$ be a closed ball in a uniformly convex and uniformly smooth Banach space $X$. Then, $P_C$ is directionally differentiable on $X$ such that, for every $v\in X$ with $v \neq \theta$, we have
	
	(i) For any $x\in (B(c, r))^o$, $P'_C(x)(v) = v;$
	
	(ii) For any $x\in X\backslash B(c, r)$,
	$$P'_C(x)(v) = \frac{r}{\|x-c\|^2}\left(\|x-c\|v - \psi\left(\frac{x-c}{\|x-c\|},\frac{v}{\|v\|}\right)\|v\|(x-c)\right);$$
	
	(iii) For any $x\in S(c, r)$, we have
	
	\ \ \ \ \ \ (a)        $P'_C(x)(v) = v-\frac{\|v\|}{r}\psi\left(\frac{x-c}{\|x-c\|},\frac{v}{\|v\|}\right)(x-c),  \ \text{if} \ v\in x_{(c, r)}^\uparrow,$
	
	\ \ \ \ \ \ (b)         $P'_C(x)(v) = v, \ \text{if} \ v\in x_{(c, r)}^\downarrow.$}

Applying Theorem 5.2 in [12] to Hilbert spaces, we have

\begin{theorem}
	Let $C = B(c, r)$ be a closed ball in Hilbert space $H$. Then, $P_C$ is directionally differentiable on $H$ such that, for every $v\in H$ with $v \neq \theta$, we have
	
	(i) For any $x\in (B(c, r))^o$, we have
	
	\ \ \ \ \ \  (a) $P'_C(x)(v) = v,$
	
	\ \ \ \ \ \  (b) $P'_C(x)(x) = x, \ \text{for} \ x \neq \theta;$
	
	(ii) For any $x\in H\backslash B(c, r)$, we have
	
	\ \ \ \ \ \ (a) $P'_C(x)(v) = \frac{r}{\|x-c\|^3}(\|x-c\|^2v-\langle x-c,v\rangle(x-c)),$
	
	\ \ \ \ \ \ (b) $P'_C(x)(x-c) = \theta;$
	
	(iii) For any $x\in S(c, r)$, we have
	
	\ \ \ \ \ \ (a)  $P'_C(x)(v) = v-\frac{1}{r^2}\langle x-c,v\rangle(x-c),  \ \text{if} \ v\in x_{(c, r)}^\uparrow;$
	
	\ \ \ \ \ \ (b)   $P'_C(x)(x-c) = \theta,$
	
	\ \ \ \ \ \ (c)   $P'_C(x)(v) = v,  \ \text{if} \ v\in x_{(c,r)}^\downarrow.$
\end{theorem}

\begin{proof}
	Notice that the norm $\|\cdot\|$ of Hilbert space $H$ is G$\hat{a}$teaux directionally differentiable. By \eqref{1.3}, for any $x, v\in H$ with $x\neq c$ and $v \neq \theta$, we have
	\begin{equation*}
		\begin{aligned}
			\psi\left(\frac{x-c}{\|x-c\|},\frac{v}{\|v\|}\right)
			=&\lim_{t\downarrow0}\frac{\left\| \frac{x-c}{\|x-c\|} + t\frac{\|v\|}{\|x-c\|} \frac{v}{\|v\|}\right\|-\left\|\frac{x-c}{\|x-c\|}\right\|}{t\frac{\|v\|}{\|x-c\|}}\\
			=& \left\langle \frac{x-c}{\|x-c\|},\frac{v}{\|v\|}\right\rangle\\
			=& \frac{1}{\|x-c\|\|v\|} \langle x-c,v\rangle.
		\end{aligned}
	\end{equation*}
	Then, this theorem follows from Theorem 5.2 in [12] immediately.
\end{proof}

\begin{corollary}
 Let $H$ be a Hilbert space. Then, $P_B$ is directionally differentiable on $H$ such that, for every $v\in H$ with $v \neq \theta$, we have

(i) For any $x\in B^o$,

\ \ \ \ \  \  (a)     $P'_B(x)(v) = v,$

\ \ \ \ \  \  (b)     $P'_B(x)(x) = x, \ \text{for} \ x \neq \theta;$

(ii) For any $x\in H\backslash B$,

\ \ \ \ \ \ 	(a) $P'_B(x)(v) = \frac{1}{\|x\|^3}(\|x\|^2v-\langle x,v\rangle x),$

\ \ \ \ \ \ 	(b) $P'_B(x)(v) = \frac{v}{\|x\|}, \ \text{if} \  x\bot v,$

\ \ \ \ \ \ 	(c) $P'_C(x)(x) = \theta;$

(iii) For any $x\in S$, we have

\ \ \ \ \ \      (a) $P'_B(x)(v) = v-\langle x,v\rangle x,  \ \text{if} \ v\in x^\uparrow,$

\ \ \ \ \ \      (b) $P'_B(x)(v) = v, \ \text{if} \ v\in x^\uparrow \ \text{and}\ x\bot v,$

\ \ \ \ \ \      (c) $P'_B(x)(v) = v,  \ \text{if} \ v\in x^\downarrow$.
\end{corollary}

\section{Directional differentiability of the metric projection operator in Hilbert spaces with orthonormal bases}
\noindent{{{5.1. \bf Closed balls in Hilbert spaces with orthonormal bases.}}}
In this section, we consider Hilbert spaces that have orthonormal bases. See [6] for more details about Hilbert spaces with orthonormal bases. Let $H$ be a Hilbert space with norm $\|\cdot\|$ and with an orthonormal basis $\{e_n\}$, which satisfies
$$\langle e_m,e_n\rangle=\left\{ \begin{aligned}
	&1, \ \text{if} \ m=n, \\
	&0,\ \text{if} \ m\neq n.
\end{aligned}\right.$$
For every $x \in H$, $x$ has the following analytic representation

$$x=\sum_{n=1}^\infty \langle x,e_n\rangle e_n$$
such that
$$\|x\|^2=\sum_{n=1}^\infty \langle x,e_n\rangle^2$$
and
$$\langle x,y\rangle =\sum_{n=1}^\infty \langle x,e_n\rangle \langle y,e_n\rangle, \ \text{for any}\ x, y\in H.$$

Recall that $B$ and $B^o$ respectively denote the closed and open unit balls and $S$ denotes the unit spheres in $H$. They are defined as follows.
$$B=\{y\in H: \sum_{n=1}^\infty \langle y,e_n\rangle^2\leq1\},$$
$$B^o = \{y\in H: \sum_{n=1}^\infty \langle y,e_n\rangle^2<1\},$$
$$S=\{y\in H: \sum_{n=1}^\infty\langle y,e_n\rangle^2 = 1\}.$$

As a special case of Corollary 4.1, the directionally derivatives of the metric projection $P_B$ enjoys the following analytic representations.

\begin{proposition}
	Let $H$ be a Hilbert space with an orthonormal basis $\{e_n\}$. Then, $P_B$ is directionally differentiable on $H$ such that, for every $v\in H$ with $v \neq \theta$, we have
	
	(i)	For any $x\in H$ with $\sum_{n=1}^\infty \langle x,e_n\rangle^2 < 1$,
	
	\ \ \ \ \      (a) $P'_B(x)(v) = v,$
	
	\ \ \ \ \    (b) $P'_B(x)(x) = x, \ \text{for} \ x \neq \theta;$\vskip0.2cm
	
	(ii) For any $x\in H$ with $\sum_{n=1}^\infty \langle x,e_n\rangle^2 > 1$,
	
	\ \ \ \ \ \     (a) $P'_B(x)(v) = \frac{1}{\|x\|^3}(\|x\|^2v -(\sum_{n=1}^\infty \langle x,e_n\rangle \langle v,e_n\rangle)x),$
	
	\ \ \ \ \ \    (b) $P'_B(x)(v) = \frac{v}{\|x\|},  \ \text{if} \ \sum_{n=1}^\infty \langle x,e_n\rangle \langle v,e_n\rangle = 0,$
	
	\ \ \ \ \ \       (c) $P'_B(x)(x) = \theta;$\vskip0.2cm
	
	(iii)   For any $x\in H$ with $\sum_{n=1}^\infty \langle x,e_n\rangle^2= 1,$
	
	\ \ \ \ \ \      (a)    $P'_B(x)(v) = v-(\sum_{n=1}^\infty \langle x,e_n\rangle \langle v,e_n\rangle)x,  \ \text{if} \ v\in x^\uparrow,$
	
	\ \ \ \ \ \     (b)    $P'_B(x)(v) = v,  \ \text{if} \ v\in x^\uparrow\ \ \text{and}\ \sum_{n=1}^\infty \langle x,e_n\rangle \langle v,e_n\rangle= 0,$
	
	\ \ \ \ \ \     (c)    $P'_B(x)(v) = v, \ \text{if} \ v\in x^\downarrow.$
\end{proposition}

\noindent{{{5.2. \bf The positive cones in Hilbert spaces with orthonormal bases.}}}
The positive cones are special closed and convex subsets in Hilbert spaces with orthonormal bases, which have been studied by many authors with applications. In this subsection, we consider the metric projection operator onto the positive cones in Hilbert spaces with orthonormal bases.

Let $H$ be a Hilbert space with an orthonormal basis $\{e_n\}$. Let $K$ denote the positive cone of $H$, which is defined by
$$K = \{y\in H: \langle y,e_n\rangle \geq 0, \ \text{for} \ n = 1, 2,\ldots\}.$$
$K$ is a pointed, closed and convex cone in $H$ with vertex at $\theta$. We write
$$K^+ = \{y\in K: \langle y,e_n\rangle > 0, \ \text{for} \ n = 1, 2, \ldots\},$$
$$ \partial K = \{y\in K: \ \text{there is at least one}\ m \ \text{such that}\ \langle y,e_m\rangle = 0\}.$$
$\partial K$ is called the boundary of $K$ (See Lemma 5.1 below. When $\{e_n\}$ is finite, $\partial K$ is the topological boundary $K$. When $\{e_n\}$ is infinite, the topological boundary of $K$ is itself). The dual cone of $K$ is denoted by $K^{\bowtie}$ satisfying
$$ K^{\bowtie}= \{z\in H: \langle z,y\rangle \leq 0, \ \text{for every}\ y \in K\}.$$
One can show that the dual cone $K^{\bowtie}$ of $K$ is also a pointed, closed and convex cone in $H$ with vertex at $\theta$, which satisfies
$$K^{\bowtie} =-K = \{z\in H: \langle z,e_n\rangle \leq 0, \ \text{for} \ n = 1, 2,\ldots\}.$$
We write
$$K^- = \{y\in K^{\bowtie}: \langle y,e_n\rangle< 0, \ \text{for} \ n = 1, 2,\ldots\}.$$
For every $y\in \partial K\backslash\{\theta\}$, we denote
$$y^\bot =\{x\in H: \langle x,y\rangle = 0\}.$$

\begin{lemma}
	Let $H$ be a Hilbert space with an orthonormal basis $\{e_n\}$.
	
	(i)\ If $\{e_n\}$ is finite, then both $K^+$ and $K^-$ are open, which are interiors of $K$ and $K^{\bowtie}$,  respectively.
	
	(ii) If $\{e_n\}$ is infinite, then neither $K^+$, nor $K^-$ is open and the interiors of $K$ and $K^{\bowtie}$ bothare empty.
\end{lemma}

\begin{proof}
	Part (i) is clear, it is because that if $\{e_n\}$ is finite, then $H$ is an Euclidean space. Next, we prove (ii). To prove $K^+$ is not open, take an arbitrary $y\in K^+$, with $\langle y,e_n\rangle > 0$, for $n = 1, 2,\ldots.$ For any $\varepsilon> 0$, we consider the neighborhood $\{z\in H: \|z-y\| < \varepsilon\}$ of $y$. Since $\langle y,e_n\rangle \rightarrow 0,$ as $n \rightarrow\infty$, there is a positive integer $k$ such that $0 < \langle y,e_k\rangle < \frac{1}{4}\varepsilon.$ Take $u \in H$ as follows
	$$\langle u,e_n\rangle =\left\{ \begin{aligned}
		& -\frac{1}{4}\varepsilon, \ \text{if} \ n=k, \\
		&\langle y,e_n\rangle,\ \text{if} \ n\neq k.
	\end{aligned}\right.$$
	This implies $\|u-y\| = \langle y,e_k\rangle + \frac{1}{4}\varepsilon < \frac{1}{2}\varepsilon.$ It follows that $u \notin\{z\in H: \|z-y\| < \varepsilon\}.$ Since $u \notin K^+$ and $\varepsilon$ is arbitrarily given small, it yields that $K^+$ is not open in $K$. We can similarly prove that $K^-$ is not open in $K^{\bowtie}$.
\end{proof}

\begin{lemma}
	Let $H$ be a Hilbert space with an orthonormal basis $\{e_n\}$. Let $K$ be the positive cone of $H$ with dual cone $K^{\bowtie}$. Then, we have
	
	(i)\ \ $P_K^{-1}(\theta) = K^{\bowtie} =-K;$
	
	(ii) For any $y\in \partial K\backslash\{\theta\}$, let $M_y = \{m\in\mathbb{N}: \langle y,e_m\rangle = 0\}.$ Then, for any $x \in H$, we have  $x \in P_K^{-1}(y)$ with $x \neq y$, if and only if
	$$\langle x,e_n\rangle =\left\{ \begin{aligned}
		&\leq0, \ \text{if} \ n\in M_y, \\
		&\langle y,e_n\rangle,\ \text{if} \ n\notin M_y,
	\end{aligned}\right.$$
	and $$0 < \sum_{m\in M_y}\langle x,e_m\rangle^2 < \infty;$$
	
	(iii) For any $y \in K^+$, we have $P_K^{-1}(y) = \{y\}.$
\end{lemma}

\begin{proof}
Proof of part (i). For $x \in H$, by the basic variational principle of $P_K$,
\begin{equation*}
	\begin{aligned}
		x \in P_K^{-1}(\theta) &\Leftrightarrow \langle x,-z\rangle \geq 0, \ \text{for every} \ z \in K\\
		&\Leftrightarrow \langle x,e_n\rangle \leq 0, \ \text{for every} \ n\\
		&\Leftrightarrow x\in K^{\bowtie} =-K.
	\end{aligned}
\end{equation*}
Next, we prove part (ii) also by the basic variational principle of $P_K$. Let $y\in \partial K\backslash\{\theta\}.$ Then there is at least one $m$ such that $\langle y,e_m\rangle = 0$. Let
$$M_y = \{m\in\mathbb{N}: \langle y,e_m\rangle = 0\}.$$

Since $y \in \partial K\backslash\{\theta\}$, it follows that $M_y$ is a nonempty proper subset of $\mathbb{N}$. Take arbitrary $\lambda_m \geq 0$, for $m\in M_y$, such that
$$0 < \sum_{m\in M_y}\lambda_m^2 < \infty.$$
Take $x \in H$ with
$$\langle x,e_n\rangle =\left\{ \begin{aligned}
	&-\lambda_m, \ \text{if} \ n\in M_y, \\
	&\langle y,e_n\rangle,\ \text{if} \ n\notin M_y.
\end{aligned}\right.$$
This implies that $x \notin K$ and
\begin{equation*}
	\begin{aligned}
		\langle x,y\rangle &=\sum_{n=1}^\infty \langle x,e_n\rangle\langle y,e_n\rangle\\
		&=\sum_{n\notin M_y}\langle x,e_n\rangle\langle y,e_n\rangle\\
		&=\sum_{n\notin M_y}\langle y,e_n\rangle\langle y,e_n\rangle\\
		&=\sum_{n=1}^\infty \langle y,e_n\rangle\langle y,e_n\rangle\\
		&=\|y\|^2.
	\end{aligned}
\end{equation*}
It follows that
\begin{equation}\label{5.1}
	\langle x-y,y\rangle = 0.
\end{equation}
And
\begin{equation*}
	\begin{aligned}
		\langle x-y,y-z\rangle&= \langle x-y,y\rangle-\langle x-y,z\rangle\\
		&=0-\sum_{m\in M_y}\langle x,e_m\rangle\langle z,e_m\rangle\\
		&=-\sum_{m\in M_y}\lambda_m\langle z,e_m\rangle\\
		&\geq0,\ \text{for every}\ z\in K.
	\end{aligned}
\end{equation*}
By the basic variational principle of $P_K$, this implies $x \in P_K^{-1}(y)$ with $x \neq y$ satisfying $(x-y)\bot y$.

On the other hand, for any $x \in P_K^{-1}(y)$ with $x \neq y$, taking $z = \frac{1}{2}y$ and $2y$, respectively, by the basic variational principle of $P_K$, we can show that $x$ satisfies \eqref{5.1}. This implies
\begin{equation}\label{5.2}
	\langle x,y\rangle = \|y\|^2,
\end{equation}
and
\begin{equation}\label{5.3}
	\langle x-y,-z\rangle \geq0, \ \text{for every}\ z\in K.
\end{equation}
In particular, take $z = e_m$, for $m\in M_y$, and $e_n$, for $n\notin M_y$, respectively, in (5.3), we have
\begin{equation}\label{5.4}
	\langle x,e_m\rangle \leq0, \ \text{for any}\ m\in M_y,
\end{equation}
and
\begin{equation}\label{5.5}
	\langle x,e_n\rangle \leq\langle y,e_n\rangle , \ \text{for any}\ n\notin M_y.
\end{equation}
By \eqref{5.2} and \eqref{5.5}, we have
\begin{equation*}
	\begin{aligned}
		\|y\|^2&=\sum_{n=1}^\infty \langle x,e_n\rangle\langle y,e_n\rangle\\
		&=\sum_{n\notin M_y}\langle x,e_n\rangle\langle y,e_n\rangle\\
		&\leq\sum_{n\notin M_y}\langle y,e_n\rangle\langle y,e_n\rangle\\
		&=\sum_{n=1}^\infty \langle y,e_n\rangle\langle y,e_n\rangle\\
		&=\|y\|^2.
	\end{aligned}
\end{equation*}
By \eqref{5.5}, this implies
\begin{equation}\label{5.6}
	\langle x,e_n\rangle =\langle y,e_n\rangle, \ \text{for any}\ n\notin M_y.
\end{equation}
Since $x \in P_K^{-1}(y)$ with $x \neq y$, there is at least one $m$ such that $\langle y,e_m\rangle = 0$. By $y \in \partial K\backslash\{\theta\}$, it follows that $M_y$ is a nonempty proper subset of $\mathbb{N}$ such that
$$ 0<\sum_{m\in M_y}\langle x,e_m\rangle^2 = \sum_{m\in \mathbb{N}}\langle x,e_m\rangle^2< \infty.$$
Proof of (iii). Let $y \in K^+$ with $y\in P_K^{-1}(y)$. Similar to the proof of part (ii), we have that $x$ satisfies \eqref{5.6}. Since $y\in K^+$, which induces $M_y = \emptyset$. This implies $x = y.$
\end{proof}

\begin{corollary}
Let $H$ be a Hilbert space with an orthonormal basis $\{e_n\}$. Let $K$ be the positive cone of $H$ with dual cone $K^{\bowtie}$. Then, for any $x \in H\backslash(K\cup K^{\bowtie})$, let $y = P_K(x)$, we have

(i)\ \ $y = P_K(x) \in \partial K\backslash\{\theta\};$

(ii) $\langle y,e_n\rangle =\left\{ \begin{aligned}
	&0, \ &\text{if} \ \langle x,e_n\rangle\leq0, \\
	&\langle x,e_n\rangle,\ &\text{if} \ \langle x,e_n\rangle
	>0,
\end{aligned}\right. $\quad for $n = 1,2,\ldots.$
\end{corollary}

\begin{proof}
For any $x \in H\backslash(K\cup K^{\bowtie})$, there are at least two positive integers $i$ and $j$ such that $\langle x,e_i\rangle < 0$ and $\langle x,e_j\rangle> 0.$ This implies that $y = P_K(x)$ defined in (ii) of this corollary satisfies $y \in \partial K\backslash\{\theta\}.$ For the given $P_K(x)$ defined in (ii), by Lemma 5.2, we have
$$x \in P_K^{-1}(y).$$
\end{proof}

\begin{theorem}
	Let $H$ be a Hilbert space with an orthonormal basis $\{e_n\}$. Let $K$ be the positive cone of $H$ with dual cone $K^{\bowtie}$. Then, we have
	
	(i)	For any $x\in K$, $P'_K(x, v) = v, \ \text{for any} \ v\in K \ \text{with} \ v \neq\theta;$
	
	(ii) For any $x\in K^{\bowtie}$, $P'_K(x, v) = \theta, \ \text{for any} \  v\in K^{\bowtie}, \ \text{with} \ v \neq\theta;$
	
	(iii) $P_K$ is directionally differentiable on $K^+$. Furthermore, for any $x \in K^+$,
	$P'_K(x)(v) = v, \ \text{for any} \ v\in H  \ \text{with} \ v \neq\theta.$
\end{theorem}

\begin{proof}
 Proof of (i). For any given $x, v\in K$ with $v \neq\theta$ and $t > 0$, we have $x+tv\in K$. It implies
\begin{equation*}
	\begin{aligned}
		P'_K(x, v)&=\lim_{t\downarrow0}{\frac{P_K(x+tv)-P_K(x)}{t}}\\
		&=\lim_{t\downarrow0}{\frac{x+tv-x}{t}}\\
		&=v.
	\end{aligned}
\end{equation*}
Proof of (ii). For any given $x, v\in K^{\bowtie}$ with $v \neq\theta$ and $t > 0$, we have $x+tv\in K^{\bowtie}.$ It implies
\begin{equation*}
	\begin{aligned}
		P'_K(x, v)&=\lim_{t\downarrow0}{\frac{P_K(x+tv)-P_K(x)}{t}}\\
		&=\lim_{t\downarrow0}{\frac{\theta-\theta}{t}}\\
		&=\theta.
	\end{aligned}
\end{equation*}
Proof of (iii). For the given $x\in K^+$ and for any given $v\in H$ with $v \neq\theta$, for any $\varepsilon > 0,$ there is $m > 1$ such that
$$\sum_{n=m+1}^\infty\langle v,e_n\rangle^2 < \varepsilon^2.$$
By $x \in K^+$, for the chosen $m > 1$, there is $t_0 > 0$ such that
$$\langle x,e_n\rangle + t\langle v,e_n\rangle > 0, \ \text{for} \ n = 1, 2, \ldots, m \ \text{and} \ 0 < t < t_0.$$
For a given $t$ with $0 < t < t_0$, set
$$ M = \{n: n > m, \langle x,e_n\rangle + t\langle v,e_n\rangle > 0\}  \ \ \text{and} \ \ N = \{n: n > m, \langle x,e_n\rangle + t\langle v,e_n\rangle \leq 0\}.$$
By part (iii) of Lemma 5.2, for $x \in K^+$, we have $P_K(x) = x.$ This implies that
\begin{equation*}
	\begin{aligned}
		&\left\|\frac{P_K(x+tv)-P_K(x)}{t}-v\right\|\\
		=&\left\| \frac{\sum_{n=1}^m(\langle x,e_n\rangle+t\langle x,e_n\rangle)e_n+ \sum_{n\in M}(\langle x,e_n\rangle+t\langle x,e_n\rangle)e_n-\sum_{n=1}^\infty\langle x,e_n\rangle e_n-t\sum_{n=1}^\infty\langle v,e_n\rangle e_n}{t}\right\|\\
		=&\left\|\frac{-\sum_{n\in N}(\langle x,e_n\rangle+t\langle x,e_n\rangle)e_n}{t}\right\|\\
		\leq&\left\|\frac{\sum_{n\in N}\langle x,e_n\rangle}{t}\right\|+\left\|\sum_{n\in N}\langle x,e_n\rangle e_n\right\|\\
		\leq&2\left\|\sum_{n\in N}\langle v,e_n\rangle e_n\right\|\\
		\leq&2\left\|\sum_{n=m+1}^\infty\langle v,e_n\rangle e_n\right\|\\
		<&2\varepsilon.
	\end{aligned}
\end{equation*}
The above second inequality is based on that, for any $n \in N$, we have $\langle x,e_n\rangle + t\langle v,e_n\rangle \leq 0.$ Since $\langle x,e_n\rangle> 0$, it follows that
$\langle x,e_n\rangle^2 \leq t^2\langle v,e_n\rangle^2.$
Since $v\in H$ with $v \neq\theta$ is arbitrarily given, the above estimations implies that $P_K$ is directionally differentiable on $K^+$, which follows that, for any $x\in K^+$, we obtain
$$P'_K(x)(v) = v, \ \text{for any} \ v\in H \ \text{with}\ v \neq\theta.$$\end{proof}

\section{Directional differentiability of the metric projection in $L^2([-\pi,\pi])$ with trigonometric orthonormal basis}
\noindent{{{6.1. \bf The trigonometric orthonormal basis in $L^2([-\pi,\pi])$.}}}
The trigonometric Fourier series play very important roles in approximation theory and optimization theory in the real Hilbert space $L^2([-\pi,\pi])$. As applications of section 5, in this section, we consider the special real Hilbert space $L^2([-\pi,\pi])$ with norm $\|\cdot\|$, which has an orthonormal basis formed by the trigonometric functions defined on $[-\pi,\pi]$:
\begin{equation}\label{6.1}
	\frac{1}{\sqrt{2\pi}} , \frac{\cos{t}}{\sqrt\pi},  \frac{\sin{t}}{\sqrt\pi},  \frac{\cos{2t}}{\sqrt\pi},  \frac{\sin{2t}}{\sqrt\pi}, \frac{\cos{3t}}{\sqrt\pi},  \frac{\sin{3t}}{\sqrt\pi},\ldots
\end{equation}
This basis is called the trigonometric orthonormal basis in $L^2([-\pi,\pi])$. For the simplicity, we denote this basis by $e_1(t), e_2(t), e_3(t), \ldots$ It follows that, for any $f \in L^2([-\pi,\pi])$, we have
\begin{equation*}
	\begin{aligned}
		&\langle f,e_1\rangle = \frac{1}{\sqrt{2\pi}}\int_{-\pi}^{\pi}{f(s)ds},\\
		&\langle f,e_n\rangle = \frac{1}{\sqrt\pi}\int_{-\pi}^{\pi}{\cos{ms}\ f(s)ds}, \ \text{if} \ n = 2m, \ \text{for any} \ m \geq 1,\\
		&\langle f,e_n\rangle= \frac{1}{\sqrt\pi}\int_{-\pi}^{\pi}{\sin{ms}\ f(s)ds}, \ \text{if} \  n = 2m + 1, \ \text{for any} \  m \geq 1.
	\end{aligned}
\end{equation*}
It implies that every $f \in L^2([-\pi,\pi])$ has the following analytic representation (almost everywhere with respect to $t \in [-\pi,\pi])$
\begin{align}\label{6.2}
	f(t) =& \sum_{n=1}^{\infty}{\left(\int_{-\pi}^{\pi}{e_n(s)f(s)ds}\right)e_n(t)}\notag \\
	=& \frac{1}{2\pi}\int_{-\pi}^{\pi}{f(s)ds}\notag\\
	&+ \frac{1}{\pi}\sum_{n=1}^{\infty}\left[\left(\int_{-\pi}^{\pi}{\cos{ns}\ f(s)ds}\right)\cos{nt}+\left(\int_{-\pi}^{\pi}{\sin{ns}\ f(s)ds}\right)\sin{nt}\right],
\end{align}
such that
$$\|f\|^2 = \frac{1}{2\pi}\left(\int_{-\pi}^{\pi}{f(s)ds}\right)^2 + \frac{1}{\pi}\sum_{n=1}^{\infty}\left(\int_{-\pi}^{\pi}{\cos{ns}\ f(s)ds}\right)^2+ \frac{1}{\pi}\sum_{n=1}^{\infty}\left(\int_{-\pi}^{\pi}{\sin{ns}\ f(s)ds}\right)^2.$$
In this case, the closed and open unit balls and the unit spheres $S$ in $L^2([-\pi,\pi])$ are respectively defined as follows
$$B =\{f \in L^2([-\pi,\pi]): \|f\|^2 \leq1\},$$ $$B^o =\{f \in L^2([-\pi,\pi]): \|f\|^2 <1\},$$ $$S =\{f \in L^2([-\pi,\pi]): \|f\|^2=1\}.$$
Then, with the above notations, by Proposition 5.1, we have
\begin{proposition}
	Let $L^2([-\pi,\pi])$ be the Hilbert space with the trigonometric orthonormal basis given by \eqref{6.1}. Then, $P_B$ is directionally differentiable on $L^2([-\pi,\pi])$ such that, for every $g\in L^2([-\pi,\pi])$ with $g \neq \theta$, we have
	
	(i)	For any $f\in L^2([-\pi,\pi])$ with $\|f\| < 1,$
	
	\ \ \ \ \ \  (a)  \  $P'_B(f)(g) = g,$
	
	\ \ \ \ \ \  (b) \ $P'_B(f)(f) = f, \ \text{for} \ f \neq\theta;$
	
	(ii)  For any $f\in L^2([-\pi,\pi])$ with $\|f\| > 1$,
	
	\ \ \ \ \ \   (a)  \ $P'_B(f)(g) = \frac{1}{\|f\|^3}(\|f\|^2g-(\sum_{n=1}^{\infty}\langle f,e_n\rangle \langle g,e_n\rangle)f),$
	
	\ \ \ \ \ \       (b) \  $P'_B(f)(g) = \frac{g}{\|f\|}, \ \text{if} \  \sum_{n=1}^{\infty}\langle f,e_n\rangle \langle g,e_n\rangle= 0,$
	
	\ \ \ \ \ \        (c) \   $P_C^\prime(f)(f) = \theta;$
	
	(iii)   For any $f\in L^2([-\pi,\pi])$ with $\|f\| = 1$, we have
	
	\ \ \ \ \ \               (a) \    $P'_B(f)(g) = g-(\sum_{n=1}^{\infty}\langle f,e_n\rangle \langle g,e_n\rangle)f,  \ \text{if} \  g\in f^\uparrow,$
	
	\ \ \ \ \ \                (b) \   $P'_B(f)(g) = g,  \ \text{if} \  g\in f^\uparrow\ \ \text{and}\ \sum_{n=1}^{\infty}\langle f,e_n\rangle \langle g,e_n\rangle= 0,$
	
	\ \ \ \ \ \                 (c)      \     $P'_B(f)(g) = g,  \ \text{if} \ g\in f^\downarrow.$
\end{proposition}

\noindent{{{6.2. \bf The positive cones in $L^2([-\pi,\pi])$ with respect to the trigonometric orthonormal basis.}}}
For the real Hilbert space $L^2([-\pi,\pi])$, we define the positive cone in $L^2([-\pi,\pi])$ with respect to the trigonometric orthonormal basis given in \eqref{6.1}. As a special case of Hilbert spaces with orthonormal bases considered in the previous section, this positive cone is a special closed and convex cone in $L^2([-\pi,\pi])$. which have been studied by many authors with applications. In this subsection, we consider the metric projection onto the positive cone in $L^2([-\pi,\pi])$ with respect to the trigonometric orthonormal basis, which is denoted by $\{e_n\}$ such that
\begin{align}\label{6.3}
	&e_1(t) =  \frac{1}{\sqrt{2\pi}} ,\notag \\
	&e_n(t) = \frac{\cos{mt}}{\sqrt\pi},   \ \text{if} \ n = 2m, \ \text{for any} \ m \geq 1,\notag \\
	&e_n(t) = \frac{\sin{mt}}{\sqrt\pi}, \ \text{if} \ n = 2m + 1, \ \text{for any} \ m \geq 1.                                         \end{align}
The positive cone in $L^2([-\pi,\pi])$ with respect to the trigonometric orthonormal basis is also denoted by $K$, which is defined by
$$K = \left\{f \in L^2([-\pi,\pi]): \langle f,e_n\rangle = \int_{-\pi}^{\pi}{e_n(s)f(s)ds}\geq 0, \ \text{for} \ n = 1, 2, \ldots\right\}.$$
$K$ is a pointed, closed and convex cone in $H$ with vertex at $\theta$. $K^+$, $\partial K$ and $K^{\bowtie} =-K$ are similarly defined as in the previous section. Since the trigonometric orthonormal basis $\{e_n\}$ listed in \eqref{6.3} in $L^2([-\pi,\pi])$ is infinite, as a consequence of Lemma 5.1, we have

\begin{corollary}
In the Hilbert space $L^2([-\pi,\pi])$ with the trigonometric orthonormal basis $\{e_n\}$ in \eqref{6.3}, we have

(i)  Neither $K^+$, nor $K^-$ is open in $L^2([-\pi,\pi])$;

(ii) The interiors of $K$ and $K^{\bowtie}$ both are empty.
\end{corollary}

We denote the sets of odd and even functions in $L^2([-\pi,\pi])$ by $O$ and $E$, respectively. Then, by the representations in \eqref{6.2}, for any $f \in L^2([-\pi,\pi])$, we have

(I)	$f \in O$ if and only if
$$f(t) =  \frac{1}{\pi}\sum_{n=1}^{\infty}{\left(\int_{-\pi}^{\pi}{\sin{ns}\ f(s)ds}\right)\sin{nt}},$$
and
$$ \langle f,e_n\rangle = \frac{1}{\sqrt\pi}\int_{-\pi}^{\pi}{\cos{ns}\ f(s)ds} = 0, \ \text{if} \ n = 2m, \ \text{for any} \ m \geq 1;$$

(II) $f \in E$ if and only if
$$f(t) = \frac{1}{\sqrt{2\pi}}\int_{-\pi}^{\pi}{f(s)ds} + \frac{1}{\pi}\sum_{n=1}^{\infty}{\left(\int_{-\pi}^{\pi}{\cos{ns}\ f(s)ds}\right)\cos{nt}},$$
and
$$ \langle f,e_n\rangle = \frac{1}{\sqrt\pi}\int_{-\pi}^{\pi}{\sin{ns}\ f(s)ds} = 0, \ \text{if} \ n = 2m + 1, \ \text{for any} \ m \geq 0.$$

\begin{lemma}
	In the Hilbert space $L^2([-\pi,\pi])$ with the trigonometric orthonormal basis $\{e_n\}$ in \eqref{6.3}, for $g \in K\backslash\{\theta\}$ with $M_g = \{m\in\mathbb{N}: \langle g,e_m\rangle = 0\}$, we have
	
	(i)\ If $g \in K\cap O$, then, $M_g\supseteq \{1, 3, 5, \ldots\}$ and for any $f \in L^2([-\pi,\pi]),\ f \in P_K^{-1}(g)$ with $f \neq g$,   if and only if
	\begin{equation*}
		\begin{aligned}
			&\frac{1}{\sqrt{2\pi}}\int_{-\pi}^{\pi}{f(s)ds}\leq 0,\\
			&\frac{1}{\sqrt\pi}\int_{-\pi}^{\pi}{\sin{ns}\ f(s)ds} \leq 0, \ \text{for} \ n\in M_g,\\
			&\frac{1}{\sqrt\pi}\int_{-\pi}^{\pi}{\cos{ns}\ f(s)ds} \leq 0, \ \text{for} \ n\in M_g,
		\end{aligned}
	\end{equation*}
	and
	$$\frac{1}{\sqrt\pi}\int_{-\pi}^{\pi}{\cos{ns}\ f(s)ds} = \frac{1}{\sqrt\pi}\int_{-\pi}^{\pi}{\cos{ns}\ g(s)ds}, \ \text{for} \ n\notin M_g;$$
	
	(ii)	If $g \in K\cap E$, then, $M_g\supseteq \{2, 4, 6, \ldots\}$ and for any $f \in L^2([-\pi,\pi]),\ f \in P_K^{-1}(g)$ with $f \neq g$, if and only if
	$$\frac{1}{\sqrt\pi}\int_{-\pi}^{\pi}{\sin{ns}\ f(s)ds} \le 0, \ \text{for} \ n\in M_g,$$
	$$\frac{1}{\sqrt\pi}\int_{-\pi}^{\pi}{\cos{ns}\ f(s)ds} \le 0,  \ \text{for} \ n\in M_g,$$
	$$\frac{1}{\sqrt\pi}\int_{-\pi}^{\pi}{\sin{ns}\ f\left(s\right)ds} = \frac{1}{\sqrt\pi}\int_{-\pi}^{\pi}{\sin{ns}\ g\left(s\right)ds},  \ \text{for} \ n\notin M_g,$$
	and
	$$\frac{1}{\sqrt{2\pi}}\int_{-\pi}^{\pi}{f(s)ds}\le 0, \ \text{if} \ \ \frac{1}{2\pi}\int_{-\pi}^{\pi}{g(s)ds}= 0,$$
	$$\frac{1}{\sqrt{2\pi}}\int_{-\pi}^{\pi}{f(s)ds} = \frac{1}{\sqrt{2\pi}}\int_{-\pi}^{\pi}{g(s)ds}, \ \text{if} \ \ \frac{1}{2\pi}\int_{-\pi}^{\pi}{g(s)ds}> 0.$$
\end{lemma}

\begin{proof}
Proof of (i). Since $g \in K\cap O$ and $M_g\supseteq \{1, 3, 5, \ldots\}$, then the fact $n\notin M_g$ implies that $n$ must be even. In (ii), by $g \in K\cap E$ and $M_g\supseteq \{2, 4, 6,\ldots\}$, it follows that $n$ must be odd for any $n\notin M_g$. Then this lemma follows from Lemma 5.2 immediately.
\end{proof}

By Lemma 6.1 and Corollary 5.1, we can obtain the analytic representations of solutions of the metric projection $P_K$.

\begin{corollary}
 Let $K$ be the positive cone of $L^2([-\pi,\pi])$ with dual cone $K^{\bowtie}$ with respect to the trigonometric orthonormal basis $\{e_n\}$ in \eqref{6.3}.  For $f \in L^2([-\pi,\pi])\backslash(K\cup K^{\bowtie})$, let $g = P_K(f)$. Then

(i)	$P_K(f)\in \partial K\backslash\{\theta\};$

(ii) For $n = 1, 2,\ldots$ we have

\ \ \ \ \ \ 	(a)\ $\langle g,e_n\rangle = 0, \ \text{if} \ \langle f,e_n\rangle\leq0,$

\ \ \ \ \ \ 	(b)\ $\langle g,e_n\rangle = \langle f,e_n\rangle, \ \text{if} \ \langle f,e_n\rangle>0.$

\noindent In particular, for odd or even functions, we have

\ \ \ \ \ \ (c) If $f$ is odd, then $g = P_K(f)$ is also odd such that
$$g(t) = \frac{1}{\pi}\sum_{n=1}^{\infty}{\left(\int_{-\pi}^{\pi}{\sin{ns}\ f(s)ds}\right)\sin{nt}},$$

\ \ \ \ \ \ (d) If $f$ is even, then $g = P_K(f)$ is also even such that
$$g(t) = \frac{1}{\sqrt{2\pi}}\int_{-\pi}^{\pi}{f(s)ds} + \frac{1}{\pi}\sum_{n=1}^{\infty}{\left(\int_{-\pi}^{\pi}{\cos{ns}\ f(s)ds}\right)\cos{nt}}.$$
\end{corollary}

\begin{proof}
Notice that, by part (a) and (b) in this corollary, in the sums in (c) and (d), they have the following properties

1) some terms may be vanished;

2) all non-vanished terms have positive coefficients, with respect to the trigonometric  orthonormal basis $\{e_n\}$ in \eqref{6.3}.

We only prove (c) and (d). If $f$ is odd, then $\langle f,e_n\rangle = 0$, for every even number $n$. By part (a), this implies $\langle g,e_n\rangle= 0$, for every even number $n$. By the representation of $g$ given in \eqref{6.2}, it follows that $g$ must be odd. Part (d) can be similarly proved.
\end{proof}

As a consequence of Theorem 5.1, we have the solutions of the directionally derivatives of the metric projection $P_K$.

\begin{corollary}
	Let $K$ be the positive cone of $L^2([-\pi,\pi])$ with dual cone $K^{\bowtie}$ with respect to the trigonometric orthonormal basis $\{e_n\}$ in \eqref{6.3}. Then, we have
	
	(i)    For any $f\in K,$
	$$P'_K(f, h) = h,\ \text{for any} \ h\in K \ \text{with} \ h \neq\theta;$$
	
	(ii)  For any $f\in K^{\bowtie}$,
	$$P'_K(f, h) = \theta, \ \text{for any} \ h\in K^{\bowtie}, \ \text{with} \ h \neq\theta;$$
	
	(iii) $P_K$ is directionally differentiable on $K^+$. More precisely speaking, for any $f \in K^+$,
	$$P'_K(f)(h) = h,  \ \text{for any} \ h\in H  \ \text{with} \ h \neq\theta.$$
\end{corollary}

\section{The metric projection in Hilbertian Bochner spaces}
\noindent{{{7.1. \bf Hilbertian Bochner spaces: Bochner spaces that are Hilbert spaces.}}}
Bochner spaces are considered as special Banach spaces of functions taking values in Banach spaces. On the other hand, Bochner spaces are generalizations of real $L_p$ spaces, which have been widely applied to function analysis and stochastic optimizations. In this section, we consider the metric projection operator and its directional differentiability in Bochner spaces that are Hilbert spaces, which are called Hilbertian Bochner spaces, in this paper.

In this subsection, we recall Hilbertian Bochner spaces. See [5, 10, 13, 21] for more details related to Bochner spaces. Let $(S, \mathcal{A}, \mu)$ be a measure space, which is assumed to be positive and complete. Let $(\mathbb{H},\|\cdot\|)$ be a real Hilbert space with inner product $\langle\cdot,\cdot\rangle.$ For any $A \in\mathcal{A}$ and $x \in \mathbb{H}$, let $1_A$ denotes the characteristic function of $A$ on the space $S$ and let $1_A\otimes x$ denote the $\mathbb{H}$-valued simple function on $S$, which is defined by

\begin{equation}\label{7.1}(1_A\otimes x)(s) = 1_A(s)\otimes x =\left\{ \begin{aligned}
		&x, \ &\text{if} \ s\in A, \\
		&\theta,\ &\text{if} \  s\notin A,
	\end{aligned}\right. \ \text{for any} \ s\in S.
\end{equation}
For an arbitrary given positive integer $n$, let $\{A_1, A_2, \ldots A_n\}$ be a finite collection of mutually disjoint subsets in $\mathcal{A}$ with $0 < \mu\left(A_i\right)<\infty,$ for all $i = 1, 2, \ldots n.$ Let $\{x_1, x_2, \ldots x_n\} \subseteq \mathbb{H}$ and let $\{a_1, a_2, \ldots a_n\}$ be a set of real numbers. Then, $\sum_{i=1}^{n}{a_i(1_{A_i}\otimes x_i)}$ is called a $\mu$-simple function from $S$ to $\mathbb{H}$ (See Definition 1.1.13 in [10]).

Let $(L_2(S; \mathbb{H}), \|\cdot\|_{L_2(S;\mathbb{H})})$ be the Lebesgue-Bochner function space called the Hilbertian Bochner space, which is the Hilbert space of $\mu$-equivalent class of strongly measurable functions, such as $f: S \rightarrow \mathbb{H}$ with norm:
$$\|f\|_{L_2(S;\mathbb{H})} = \left(\int_S\|f(s)\|^2d\mu(s)\right)^\frac{1}{2} < \infty.$$
The inner product in this Hilbertian Bochner space $(L_2(S; \mathbb{H}), \|\cdot\|_{L_2(S;\mathbb{H})})$ is simply denoted by $\langle\cdot,\cdot\rangle_{L_2}$. By Theorem 2.2 and Proposition 2.3 in [13], $L_2(S; \mathbb{H})$ is a Hilbert space and it has the following properties. For an arbitrary given $A \in\mathcal{A}$ with $0<\mu\left(A\right)<\infty \ \text{and for any}\ x,y\in\mathbb{H}$, we have

(i)\ \ \ $\frac{1}{{\mu\left(A\right)}^\frac{1}{2}}\left(1_A\otimes x\right)\in L_2(S; \mathbb{H});$

(ii) \ \  $\|\frac{1}{\mu(A)^\frac{1}{2}}(1_A\otimes x)\|_{L_2(S;\mathbb{H})} = \|x\|;$

(iii) \   $ \|\frac{1}{\mu(A)^\frac{1}{2}}(1_A\otimes x)\pm\frac{1}{\mu(A)^\frac{1}{2}}(1_A\otimes y)\|_{L_2(S;\mathbb{H})} = \|x\pm y\|;$

(iv) \ The mapping $x \rightarrow \frac{1}{{\mu\left(A\right)}^\frac{1}{2}}\left(1_A\otimes x\right)$ (isometric) embeds $H$ into $L_2(S; \mathbb{H}).$

\noindent{{{7.2. \bf The metric projection onto closed balls in Hilbertian Bochner spaces.}}}
Let $L_2(S; \mathbb{H})$ be the Hilbertian Bochner space discussed in the previous subsection. Recall
that the closed unit ball $B$ in $L_2(S; \mathbb{H})$ is defined as follows
$$B =\left\{ f\in L_2(S; \mathbb{H}): \int_S \|f(s)\|^2d\mu(s)\leq1\right\}.$$

By Corollary 4.1, the directionally derivatives of the metric projection $P_B$ has the following analytic representations.

\begin{proposition}
	Let $L_2(S; \mathbb{H})$ be a Hilbertian Bochner space. Then, $P_B$ is directionally differentiable on $L_2(S; \mathbb{H})$ such that, for every $h\in L_2(S; \mathbb{H})$ with $h \neq \theta$, we have
	
	(i)\	\ For any $f\in L_2(S; \mathbb{H})$ with $\int_S \|f(s)\|^2d\mu(s) < 1$, we have
	
	\ \ \ \ \ \ 	(a)   $P'_B(f)(h) = h,$
	
	\ \ \ \ \ \ 	(b)   $P'_B(f)(f) = f, \ \text{for} \ f \neq \theta;$
	
	(ii) \  For any $f\in L_2(S; \mathbb{H})$ with $\int_S \|f(s)\|^2d\mu(s) > 1,$ we have
	
	\ \ \ \ \ \  (a)      $P'_B(f)(h) = \frac{1}{\|f\|_{L_2(S;\mathbb{H})}^3}\left(\|f\|_{L_2(S;\mathbb{H})}^2h-\int_S \langle h(s),f(s)\rangle d\mu(s)f\right),$
	
	\ \ \ \ \ \ 	   (b)      $P'_B(f)(h) = \frac{h}{\|f\|_{L_2(S;\mathbb{H})}}, \ \text{if} \  \int_S \langle h(s),f(s)\rangle d\mu(s) = 0,$
	
	\ \ \ \ \ \           (c) $P'_B(f)(f) = \theta;$
	
	(iii) \  For any $f\in L_2(S; \mathbb{H})$ with $\int_S \|f(s)\|^2d\mu(s) = 1,$ we have
	
	\ \ \ \ \ \ 	         (a)    $P'_B(f)(h) = h-\left(\int_S \langle h(s),f(s)\rangle d\mu(s)\right) f,  \ \text{if} \ h\in f^\uparrow,$
	
	\ \ \ \ \ \ 	            (b)     $P'_B(f)(h)  = h,   \ \text{if} \ h\in h^\uparrow \ \text{and}  \ \int_S \langle h(s),f(s)\rangle d\mu(s) = 0,$
	
	\ \ \ \ \ \ 	                           (c)     $P'_B(f)(h) = h,  \ \text{if} \  h\in f^\downarrow.$
\end{proposition}

\noindent{{{7.3. \bf Hilber Spaces with orthonormal bases.}}}
In this subsection, we consider a Hilbertian Bochner space $L_2(S; \mathbb{H})$, in which $(S, \mathcal{A}, \mu)$ is a probability space $(\mu(S) =1)$ and the Hilbert space $\mathbb{H}$ has an orthonormal basis$\{b_n\}$. Then, for any $f\in L_2(S; \mathbb{H})$, $f$ has the following point wisely analytic representation
\begin{equation}\label{7.2}
	f(s) = \sum_{n=1}^\infty \langle f(s),b_n\rangle b_n, \ \text{for} \ \mu\text{-almost all } \ s \in S,
\end{equation}
such that
\begin{align}\label{7.3}
	\|f\|_{L_2(S;\mathbb{H})}^2 &=\int_S \left\langle\sum_{n=1}^\infty \langle f(s),b_n\rangle b_n, \sum_{n=1}^\infty \langle f(s),b_n\rangle b_n\right\rangle d\mu(s)\notag \\
	&=\int_S \left(\sum_{n=1}^\infty \langle f(s),b_n\rangle^2\right)d\mu(s)\notag \\
	&=\sum_{n=1}^\infty\int_S \langle f(s),b_n\rangle^2d\mu(s).
\end{align}
And, for any $f, g\in L_2(S; \mathbb{H})$, we have
\begin{align}\label{7.4}
	\langle f,g\rangle _{L_2} &= \int_S \langle f(s),g(s)\rangle d\mu(s)\notag\\
	&= \int_S\left\langle\sum_{n=1}^\infty \langle f(s),b_n\rangle b_n,\sum_{n=1}^\infty \langle g(s),b_n\rangle b_n\right\rangle d\mu(s)\notag\\
	&=\int_S\left(\sum_{n=1}^\infty \langle f(s),b_n\rangle \langle g(s),b_n\rangle\right) d\mu(s)\notag\\
	&=\sum_{n=1}^\infty\int_S \langle f(s),b_n\rangle\langle g(s),b_n\rangle d\mu(s).
\end{align}
Since every element in $L_2(S; \mathbb{H})$ has the analytic representation \eqref{7.2}, we introduce the pointwise positive cone in $L_2(S; \mathbb{H})$ by
\begin{equation}\label{7.5}
	K = \{g\in L_2(S; \mathbb{H}): \langle g(s),b_n\rangle \geq 0, \ \text{for} \  \mu\text{-almost all } \ s\in S, \ \text{for} \ n = 1, 2,\ldots\}.
\end{equation}
One can show that $K$ is a pointed, closed and convex cone in this Hilbert space $L_2(S; \mathbb{H})$ (with vertex at $\theta$). $K$ is called the pointwise positive cone of $L_2(S; \mathbb{H})$ with respect to the orthonormal basis$\{b_n\}$ in $\mathbb{H}.$

\begin{proposition}
	Let $L_2(S; \mathbb{H})$ be a Hilbertian Bochner space, in which $(S, \mathcal{A}, \mu)$ is a probability space and the Hilbert space $\mathbb{H}$ has an orthonormal basis$\{b_n\}$. Let $K$ be the pointwise positive cone of $L_2(S; \mathbb{H})$ with respect to the orthonormal basis$\{b_n\}$ of $\mathbb{H}$. Then, for any $f\in L_2(S; \mathbb{H})$ and  $g\in K,$ we have that $g = P_K(f)$ if and only if
	\begin{equation}\label{7.6}\langle g(s),b_n\rangle =\left\{ \begin{aligned}
			&0, \ &\text{if} \ \langle f(s),b_n\rangle<0, \\
			&\langle f(s),b_n\rangle,\ &\text{if} \  \langle f(s),b_n\rangle\geq0,
		\end{aligned}\right. \ \text{for}\ \mu\text{-almost all}\ s \in S,\ \text{for}\ n = 1, 2, \ldots.
\end{equation}
\end{proposition}

\begin{proof}
Let
$$g(s) =\sum_{n=1}^\infty \langle g(s),b_n\rangle b_n,\ \text{for}\ \mu\text{-almost all} \ s \in S.$$
By \eqref{7.6}, we have $g\in K.$ For any $h\in K$ with
$$ h(s) =\sum_{n=1}^\infty \langle h(s),b_n\rangle b_n,\ \text{for}\ \mu\text{-almost all} \ s \in S,$$
we calculate
\begin{align}\label{7.7}
	\langle f-g,g-h\rangle_{L_2}\notag
	=& \int_S \langle f(s)-g(s),g(s)-h(s)\rangle d\mu(s)\notag\\
	=& \int_S\left\langle\sum_{n=1}^\infty \langle f(s)-g(s),b_n\rangle b_n,\sum_{n=1}^\infty \langle g(s)-h(s),b_n\right\rangle b_n\rangle d\mu(s)\notag\\
	=&\int_S\sum_{n=1}^\infty \langle f(s)-g(s),b_n\rangle \langle g(s)-h(s),b_n\rangle d\mu(s)\notag\\
	=& \sum_{n=1}^\infty\int_S \langle f(s)-g(s),b_n\rangle \langle g(s)-h(s),b_n\rangle d\mu(s).
\end{align}

For each fixed $n$, by \eqref{7.6},
\begin{equation*}
	\begin{aligned}
		&\int_S \langle f(s)-g(s),b_n\rangle \langle g(s)-h(s),b_n\rangle d\mu(s)\\
		=& \int_{\langle f(s),b_n\rangle<0} \langle f(s)-g(s),b_n\rangle \langle g(s)-h(s),b_n\rangle d\mu(s)\\
		&+  \int_{\langle f(s),b_n\rangle\geq0} \langle f(s)-g(s),b_n\rangle \langle g(s)-h(s),b_n\rangle d\mu(s)\\
		=& \int_{\langle f(s),b_n\rangle<0} \langle f(s),b_n\rangle \langle -h(s),b_n\rangle d\mu(s)\\
		&+  \int_{\langle f(s),b_n\rangle\geq0} \langle f(s)-f(s),b_n\rangle \langle f(s)-h(s),b_n\rangle d\mu(s)\\
		=& \int_{\langle f(s),b_n\rangle<0} \langle f(s),b_n\rangle \langle -h(s),b_n\rangle d\mu(s)\\
		\geq& 0, \ \text{for any} \ h\in K.
	\end{aligned}
\end{equation*}
By the basic variational principle of $P_K$, this implies $g = P_K(f)$. Since $K$ is a pointed, closed and convex cone in the Hilbert space $L_2(S; \mathbb{H})$, $P_K(f)$ uniquely exists.
\end{proof}

For the pointwise positive cone $K$ of $L_2(S; \mathbb{H})$, $K^+$, $\partial K $ and $K^{\bowtie}=-K$ are similarly defined as in the previous sections.

$$K^+ = \{g\in K: \langle g(s),b_n\rangle> 0, \  \text{for}\ \mu\text{-almost all} \ s \in S, \ \text{for} \ n = 1, 2,\ldots\},$$
$$\partial K = \{g\in K: \  \text{there is} \ n \in\mathbb{N} \  \text{such that} \ \mu\{s \in S: \langle g(s),b_n\rangle = 0\} > 0,$$
$$K^{\bowtie}=-K = \{g\in L_2(S; \mathbb{H}): \langle g(s),b_n\rangle \leq0, \  \text{for}\ \mu\text{-almost all}\ s \in S, \ \text{for} \ n = 1, 2,\ldots\}.$$

The following corollary is an immediately consequence of Proposition 7.2.

\begin{corollary}
Let $L_2(S; \mathbb{H})$ be a Hilbertian Bochner space with the pointwise positive cone $K$ with respect to the orthonormal basis $\{b_n\}$ of $\mathbb{H}.$ Then

(i)\ \	For any $g\in K,$ we have $P_K(g) = g;$

(ii)\ $P_K^{-1}(\theta) = K^{\bowtie} =-K;$

(iii)\ For any $g \in \partial K\backslash\{\theta\}, \ f \in L_2(S; \mathbb{H})$, we have that $f \in P_K^{-1}(g)$ with $f \neq g$ if and only if
\begin{equation*}\langle f(s),b_n\rangle =\left\{ \begin{aligned}
		&<0, \ &\text{if} \ \langle g(s),b_m\rangle=0, \\
		&\langle g(s),b_n\rangle,\ &\text{if} \  \langle g(s),b_m\rangle\geq0,
	\end{aligned}\right. \ \text{for}\ \mu\text{-almost all}\ s \in S,\ \text{for}\ n\in\mathbb{N};
\end{equation*}

(iv)\	For any $g \in K^+$, we have $P_K^{-1}(g) = \{g\}.$
\end{corollary}

\begin{remark}
	For any $g \in \partial K\backslash\{\theta\}, \ f \in L_2(S; \mathbb{H}),$ suppose that $f \in P_K^{-1}(g)$ with $f \neq g$. From the representation in (iii), for any given $n\in \mathbb{N}$, we have
	$$\mu\{s \in S: \langle g(s),b_n\rangle = 0\} = 0   \Rightarrow   \langle f(s),b_n\rangle = \langle g(s),b_n\rangle, \ \text{for}\ \mu\text{-almost all}\ s \in S.$$
\end{remark}

\begin{proposition}
	Let $L_2(S; \mathbb{H})$ be a Hilbertian Bochner space. Let $K$ be the pointwise positive cone in $L_2(S; \mathbb{H})$ with respect to the orthonormal basis$\{b_n\}$ of $\mathbb{H}$. Then, we have
	
	(i)     For any $f\in K,$
	$P'_K(f, h) = h, \ \text{for any} \ h\in K \ \text{with} \ h \neq\theta;$
	
	(ii)   For any $f\in K^{\bowtie},$
	$P'_K(f, h) = \theta, \ \text{for any} \ h\in K^{\bowtie}, \ \text{with} \ h \neq\theta;$
	
	(iii)    $P_K$ is directionally differentiable on $K^+$. More precisely speaking, for any $f \in K^+$,
	
	$ P'_K(f)(h) = h, \ \text{for any} \ h\in H \ \text{with} \ h \neq\theta.$
\end{proposition}

\noindent{{{7.4. \bf  Hilbertian Bochner spaces with orthonormal systems.}}}
Following the notations given in \eqref{7.1}, for every $b_n$, we define the simple function in $L_2(S; \mathbb{H})$ as follows:
$$
(1_S\otimes b_n)(s) = 1_S(s)\otimes b_n = b_n, \ \text{for any} \ s \in S.$$

It follows that, for every $n$, $1_S\otimes b_n$ is a constant function defined on $S$ with value $b_n$.

\begin{proposition}
	Let $L_2(S; \mathbb{H})$ be a Hilbertian Bochner space, in which $(S, \mathcal{A}, \mu)$ is a probability space and the Hilbert space $\mathbb{H}$ has an orthonormal basis$\{b_n\}$. Then, $\{1_S\otimes b_n\}$ is an orthonormal system in $L_2(S; \mathbb{H})$. However, $\{1_S\otimes b_n\}$ is not a basis of $L_2(S; \mathbb{H})$.
\end{proposition}

\begin{proof}
By $\mu(S) =1,$ we calculate
\begin{equation*}
	\begin{aligned}
		\langle1_S\otimes b_n,1_S\otimes b_m\rangle_{L_2}=&\int_S \langle (1_S\otimes b_n)(s),(1_S\otimes b_m)(s)\rangle d\mu(s)\\
		=& \int_S \langle b_n,b_m\rangle d\mu(s)\\
		=&\left\{ \begin{aligned}
			&1, \ &\text{if} \ n=m, \\
			&0, \ &\text{if} \ n\neq m.
		\end{aligned}\right.
	\end{aligned}
\end{equation*}
This implies that $\{1_S\otimes b_n\}$ is an orthonormal system in $L_2(S; \mathbb{H})$. Next, we show that $\{1_S\otimes b_n\}$ is not a basis of $L_2(S; \mathbb{H})$. We chose $A \in \mathcal{A}$ with $\mu(A) = \frac{1}{2}$ and define $G: S \rightarrow\mathbb{R}$ by
\begin{equation*}G(s) =\left\{ \begin{aligned}
		&1, \ &\text{for} \ s\in A , \\
		&-1,\ &\text{for} \  s\notin A,
	\end{aligned}\right. \ \text{for all}\  s \in S.
\end{equation*}
Then, we define $f\in L_2(S; \mathbb{H})$ as follows
$$\begin{aligned}
	f(s) &= \sum_{n=1}^{\infty}{{\frac{1}{2^n} G(s)}b_n}\notag\\
	&=\sum_{n=1}^{\infty}{{\frac{1}{2^n} G(s)}\left(1_S\otimes b_n\right)\left(s\right)},\ \text{for} \ s\in S.
\end{aligned}$$
It is clear that $f \neq\theta$. For any $m \in\mathbb{N}$, we have
\begin{equation*}
	\begin{aligned}
		\langle f,1_S\otimes b_m\rangle_{L_2} &= \int_S \left\langle \sum_{n=1}^\infty\frac{1}{2^n} G(s)(1_S\otimes b_n)(s),1_S(s)\otimes b_m\right\rangle d\mu(s)\\
		&=\int_S \left\langle \sum_{n=1}^\infty\frac{1}{2^n} G(s) b_n, b_m\right\rangle d\mu(s)\\
		&=\int_S \left\langle \frac{1}{2^n} G(s) b_m, b_m\right\rangle d\mu(s)\\
		&=\frac{1}{2^n}\int_S G(s)d\mu(s)\\
		&= 0.
	\end{aligned}
\end{equation*}
Since $f \neq\theta.$ This implies that $\{1_S\otimes b_n\}$ is not a basis of $L_2(S; \mathbb{H})$.
\end{proof}

Let $D$ denote the proper subspace of $L_2(S; \mathbb{H})$ expanded by the orthonormal system $\{1_S\otimes b_n\}$ in $L_2(S; \mathbb{H})$. Then, for every $g\in D,$ it enjoys the following analytic representation:
$$\begin{aligned}
	g(s)&= \sum_{n=1}^\infty\langle g,1_S\otimes b_n\rangle_{L_2}(1_S\otimes b_n)(s)\notag\\
	&=\sum_{n=1}^\infty\int_S \langle g(t),b_n\rangle d\mu(t)(1_S\otimes b_n)(s)\notag\\
	&=\sum_{n=1}^\infty\int_S \langle g(t),b_n\rangle d\mu(t) b_n, \ \text{for}\ \mu\text{-almost all}\ s \in S.
\end{aligned}\eqno (7.8)$$

\noindent{\bf Observations} (a)  The equations (7.8) implies that every element in $D$ is a constant function, it is because that every $1_S\otimes b_n$ in the orthonormal system $\{1_S\otimes b_n\}$ is a constant function;
(b) $\{1_S\otimes b_n\}$  is an orthonormal basis of $D$, which is a proper subspace of $L_2(S; \mathbb{H})$.

Notice that, in the considered Hilbertian Bochner space $L_2(S; \mathbb{H})$, the measure space $(S, \mathcal{A}, \mu)$ is a probability space. This implies that, for any $f\in L_2(S; \mathbb{H})$, $f$ is Bochner integrable on $(S, \mathcal{A}, \mu)$. That is,
$$f\in L_2(S; \mathbb{H})\Rightarrow  \int_{S}{f(t)d\mu(t)} \ \text{exists and it is in} \ \mathbb{H}.$$
Following the notations of the expected values of random variables in probability theory, for any given $f\in L_2(S; \mathbb{H})$, we write
$$E(f) = \int_{S}{f(t)d\mu(t)}\in \mathbb{H}.$$

\begin{theorem}
	Let $L_2(S; \mathbb{H})$ be a Hilbertian Bochner space, in which $(S, \mathcal{A}, \mu)$ is a probability space and the Hilbert space $\mathbb{H}$ has an orthonormal basis $\{b_n\}$. Let $D$ be the proper subspace of $L_2(S; \mathbb{H})$ with an orthonormal basis $\{1_S\otimes b_n\}$. Then, for any $f\in L_2(S; \mathbb{H}),$ we have
	\begin{equation*}
		\begin{aligned}
			P_D(f)&= 1_S\otimes E(f)\\
			&= \sum_{n=1}^\infty\langle E(f),b_n\rangle(1_S\otimes b_n)\\
			&=\sum_{n=1}^\infty\langle 1_S\otimes E(f),1_S\otimes b_n\rangle_{L_2}(1_S\otimes b_n).
		\end{aligned}
	\end{equation*}
\end{theorem}
	
\begin{proof}
	Notice that $1_S\otimes E(f)$ is a constant functional satisfying
	$$\int_{S}{(1_S\otimes E(f))(t))d\mu(t)}= E\left(f\right).$$
	For any $g\in D$, we calculate
	$$\begin{aligned}
		&\langle f-(1_S\otimes E(f)),(1_S\otimes E(f))-g\rangle_{L_2}\notag\\
		=& \int_S \langle f(t)-(1_S\otimes E(f))(t),(1_S\otimes E(f))(t)-g(t)\rangle d\mu(t)\notag\\
		=&\left\langle\int_S f(t)-(1_S\otimes E(f))(t) d\mu(t),\int_S((1_S\otimes E(f))(t)-g(t))d\mu(t)\right\rangle\notag\\
		=&\left\langle\int_S (f(t)d\mu(t)-E(f),\int_S((1_S\otimes E(f))(t)-g(t))d\mu(t)\right\rangle\notag\\
		=&\left\langle\theta,\int_S ((1_S\otimes E(f))(t)-g(t))d\mu(t)\right\rangle\notag\\
		=&0, \ \text{for every}\ g\in D.
	\end{aligned}\eqno (7.9)$$
	By the basic variational principle of $P_D$, (7.9) implies $P_D(f) = 1_S\otimes E(f)$. Since $\{b_n\}$ is an orthonormal basis of the Hilbert space $\mathbb{H}$, for $E(f)\in \mathbb{H}$, we have
	$$E(f) =\sum_{n=1}^\infty \langle E(f),b_n\rangle b_n.$$
	By property (iv) of Bochner spaces recalled in subsection 7.1, this implies
	$$1_S\otimes E\left(f\right)=\sum_{n=1}^\infty \langle E(f),b_n\rangle (1_S\otimes b_n).$$
	It is clear that $\langle 1_S\otimes E(f),1_S\otimes b_n\rangle_{L_2}=\langle E(f),b_n\rangle$, for all $n$. It follows that
	$$\sum_{n=1}^\infty \langle E(f),b_n\rangle (1_S\otimes b_n)=\sum_{n=1}^\infty \langle 1_S\otimes E(f),1_S\otimes b_n\rangle_{L_2}(1_S\otimes b_n).$$
	This completes the proof of this theorem.
\end{proof}
	
We apply the results in Theorem 7.1 to study the directionally differentiability of the metric projection on the proper subspace $D$.
	
\begin{theorem}
Let $L_2(S; \mathbb{H})$ be a Hilbertian Bochner space, in which $(S, \mathcal{A}, \mu)$ is a probability space and the Hilbert space $\mathbb{H}$ has an orthonormal basis$\{b_n\}$. Let $D$ be the proper subspace of $L_2(S; \mathbb{H})$ with an orthonormal basis $\{1_S\otimes b_n\}$. Then, $P_D$ is directional differentiable on $L_2(S; \mathbb{H}).$
Furthermore, for any $f\in L_2(S; \mathbb{H})$, we have
$$P'_D(f)(h) = 1_S\otimes E(h),\ \text{for any}\ h\in L_2(S; \mathbb{H}) \ \text{with} \ h \neq\theta.$$
\end{theorem}
	
\begin{proof}
For arbitrary given $f\in L_2(S; \mathbb{H})$ and for any $h\in L_2(S; \mathbb{H})$ with $h \neq\theta$, by Theorem 7.1, we have
\begin{equation*}
	\begin{aligned}
		P'_D(f)(h)&=\lim_{t\downarrow0}{\frac{P_D\left(f+th\right)-P_D\left(f\right)}{t}}\\
		&= \lim_{t\downarrow0}{\frac{1_S\otimes E(f+th)-(1_S\otimes E(f))}{t}}\\
		&=\lim_{t\downarrow0}{\frac{1_S\otimes (E(f)+E(th))-(1_S\otimes E(f))}{t}}\\
		&=\lim_{t\downarrow0}{\frac{1_S\otimes E(f)+(1_S\otimes E(th))-(1_S\otimes E(f))}{t}}\\
		&=\lim_{t\downarrow0}{\frac{1_S\otimes E(th)}{t}}\\
		&=\lim_{t\downarrow0}{\frac{1_S\otimes (tE(h))}{t}}\\
		&=\lim_{t\downarrow0}{\frac{t(1_S\otimes E(h))}{t}}\\
		&=1_S\otimes E(h).
	\end{aligned}
\end{equation*}
\end{proof}

\vskip 6mm
\noindent{\bf Acknowledgments}

\noindent   The authors thank Professor Simeon Reich for his kind communications in the stage of development of this paper. Author Li Cheng is very grateful to the National Natural Science Foundation of China (12171217) and (11701246) for partially financial support for this work. Author Lishan Liu is very grateful to the National Nurture Science Foundation of China (11871302), the Natural Science Foundation of Shandong Province of China (ZR2022MA009) and  ARC Discovery Project Grant (DP230102079) for partially financial support for this work.
Author Xie thanks the National Nurture Science Foundation of China (12171217) and (11771194) for partially financial support for this work.

\end{document}